\documentclass[12pt,a4paper]{amsart}
\usepackage{amsaddr}
\usepackage{amsmath,amscd}
\usepackage{latexsym,amsmath,amssymb,mathrsfs,setspace,enumerate,tikz}
\usepackage{graphicx}
\usepackage{subcaption}
\usepackage[T1]{fontenc}
\usepackage[utf8]{inputenc}
\usepackage{lmodern}
\usepackage{amssymb}
\usepackage{multicol}
\usepackage{tikz-network}
\usetikzlibrary{automata,positioning}
\usepackage[english]{babel} 
\usepackage{blindtext}

\theoremstyle{plain}
\newtheorem{theorem}{Theorem}[section]

\newtheorem{proposition}[theorem]{Proposition}

\theoremstyle{definition}
\newtheorem{definition}[theorem]{Definition}

\newtheorem{remark}[theorem]{Remark}
\newtheorem{example}[theorem]{Example}
 
\DeclareMathOperator{\dom}{dom}

\begin{document}
	
	\title[Ordered Semigroups and categories of ideals]{Ordered semigroups and ideals categories of  principal ideal rings}
	
			\author[ P. K. Minnumol and P. G. Romeo]{P. K. Minnumol and P. G. Romeo}
			
	\address{Department of Mathematics\\
		Cochin University of Science and Technology(CUSAT)\\Kochi, Kerala, India, 682022}
	\email{$pkminnumol@gmail.com,\, romeo_-parackal@yahoo.com$}
	
%\author[Minnumol P K]{Minnumol P K}
%\address{Research Scholar\\Department of Mathematics\\
%	Cochin University of Science and Technology(CUSAT)\\Kochi, Kerala, India, 682022}
	\thanks{First author wishes to thank Cochin University of Science And Technology for providing financial support.\\
		Second author wishes to thank Council of Scientific and Industrial Research(CSIR) INDIA, for providing financial support.
		 }

\begin{abstract}
	In this paper, we view the collection of ideals of a commutative principal ideal ring from two perspectives: one as an ordered semigroup $\mathcal{I}(R)$ and the other as a category $\mathbb{I}_R$. It is shown that $\mathcal{I}(R)$ is a regular ordered semigroup whereas $\mathbb{I}_R$ is a category with subobjects. Further we establish the correspondence between these structures.
\end{abstract}
\maketitle
An ordered semigroup is a semigroup equipped with a compatible partial order, there are various investigations carried out in ordered semigroups. 
Some authors have studied partial orders on regular semigroups; an ordered semigroup of this type is called an \emph{ordered regular semigroup}. 
On the other hand, some authors have studied regularity within ordered semigroups; such a notion is referred to as \emph{ordered regularity}, and semigroups in which every element is ordered regular are called \emph{regular ordered semigroups}. 
Since in an ordered semigroup every regular element is ordered regular the class of regular ordered semigroups is much larger than the class of ordered regular semigroups. 
Here we construct the ordered semigroup structure as well as the categorical structure of ideals of a commutative principal ideal ring.  We establish the correspondence between these structures.

\section{Preliminaries}
First we recall some definitions and basic results regarding ordered semigroups, rings and category theory needed in this paper.
\subsection{Ordered semigroups}

An \textit{ordered semigroup}  $(S,\cdot,\leq)$ is a partially ordered set $(S,\leq)$ together with a semigroup structure $(S,\cdot)$ such that the order is compatible with multiplication; \cite{txtbk}. 
For a subset $A \subseteq S$, the set
\[
(A] = \{x \in S \mid (\exists a \in A)\; x \leq a\}
\]
is called the \textit{downward closure} of $A$. A subset $A$ is said to be \textit{downward closed} if $(A] = A$.

\begin{definition}\cite{keharegularity}
	An element $a \in S$ is said to be \textit{ordered regular} if there exists $x \in S$ such that
	$a \leq axa$ or equivalently $a \in (aSa]$, ordered regular elements of $S$ is denoted by $Reg_{\leq}(S)$
\end{definition}
When $Reg_{\leq}(S)=S$, then $S$ is called a \textit{regular ordered semigroup}.
\begin{example}
Consider $(\mathbb{N},\cdot,\leq)$, where $\mathbb{N}$ denotes the set of natural numbers with the usual multiplication and order. Clearly, $1$ is the only regular element in the classical sense, but for every $n \in \mathbb{N}$, there exists $x \in \mathbb{N}$ such that $n \leq nxn$, hence, every element of $\mathbb{N}$ is ordered regular and $(\mathbb{N},\cdot,\leq)$ is a regular ordered semigroup.
\end{example}

An ordered semigroup $S$ is called \textit{intra-regular}  if for each $a \in S$ there exist $x,y \in S$ such that
$a \leq xa^{2}y$
or equivalently $a \in (Sa^{2}S]$, \cite{intrregular}.
An element $a' \in S$ is said to be an \textit{ordered inverse}  of $a$ if
$a \leq aa'a \quad \text{and} \quad a' \leq a'aa'.$
The set of all ordered inverses of $a$ is denoted by $V_{\leq}(a)$, \cite{jamadar}.
An ordered semigroup $S$ is called an \textit{inverse ordered semigroup} if for each $a \in S$, any two ordered inverses of $a$ are $\mathcal{H}$-related.
An element $e \in S$ is called an \textit{ordered idempotent} if $e \leq e^{2}$.

The following are some important classes of regular ordered semigroups.
\begin{itemize}
	\item an ordered semigroup $S$ is \textit{group-like} if for all $a,b \in S$ there exist $x,y \in S$ such that 	$a \leq xb$ and $a \leq by$.
	\item an ordered semigroup $S$ is \textit{completely regular} if for every $a \in S
	a \in (a^{2}Sa^{2}].$
	\item a regular ordered semigroup $S$ is called a \textit{Clifford ordered semigroup} if for all $a \in S$ and $e \in E_{\leq}(S)$ there exist $u,v \in S$ such that
	$ae \leq eua \quad \text{and} \quad ea \leq ave.$
	
\end{itemize} 
A nonempty subset $A$ of $S$ is called a \textit{left ideal} ( \textit{right ideal}) of $S$ if
$SA \subseteq A$ ($AS \subseteq A$) and $(A]=A$.
A subset $A$ is called an \textit{ideal} of $S$ if it is both a left and a right ideal.
The principal left ideal, principal right ideal and principal ideal generated by $a \in S$ are denoted by $L(a)$, $R(a)$, and $I(a)$, respectively and are defined in the usual manner \cite{kahaloideal}.
\begin{center}
	$ L(a) = (a \cup Sa] $\\
	$ R(a) = (a \cup aS] $\\
	$ I(a) = (a \cup Sa \cup aS \cup SaS] $\\
	
\end{center}
The \textit{Green's relations}  $ \mathscr{L} ,\mathscr{R}, \mathscr{J}, \mathscr{H}, \mathscr{D}  $ on an ordered semigroup $ S $ are given as follows:
\begin{center}
	$ a\,\mathscr{L} \,b $ if and only if $ L(a) = L(b) $\\
	$ a\,\mathscr{R} \,b $ if and only if $ R(a) = R(b) $\\
	$ a\,\mathscr{J} \,b $ if and only if $ I(a) = I(b) $\\
	$ \mathscr{H} = \mathscr{L} \cap \mathscr{R} $\\
	$ \mathscr{D} = \mathscr{L} \vee \mathscr{R} $
\end{center}
If $ S $ is a regular ordered  semigroup for each $ a \in{S} $ there exists $ x \in{S} $ such that $ a \leq axa$, hence $ a \in (Sa] $ and 
$ a \in (aS] $ and so the Green's equivalence for regular ordered semigroups are:
\begin{center}
	$ a\,\mathscr{L} \,b $ if and only if $ (Sa]  = (Sb] $\\
	$ a\,\mathscr{R} \,b $ if and only if $ (aS]  = (bS] $\\
\end{center}
The following theorem characterizes Green’s relations in a regular ordered semigroup.
\begin{proposition}(cf.\cite{mpk})
	Let $ a $ and $ b $ be elements of a regular ordered semigroup $ S $. Then 	$ a\,\mathscr{L} \,b $ if and only if there exist $ x,y \in{S} $ such that $ a \leq xb $ and $ b \leq ya $. Also $ a\,\mathscr{R} \,b $ if and only if there exist $ u,v \in{S} $ such that $ a \leq bu $ and $ b \leq av $.
\end{proposition}
The inverse transversals for regular ordered semigroups. is defined below.  
\begin{definition}\label{definvtrsl}(cf.\cite{mpkinverse})
	Let $(S,\cdot,\leq)$ be a regular ordered semigroup. A subset $S_{0}\subseteq S$
	is called an inverse transversal\index{Inverse transversal} of $S$ if the following conditions hold:
	\begin{enumerate}
		\item $S_{0}$ is an inverse ordered subsemigroup of $S$;
		\item $(S_{0}] = S_{0}$;
		\item $V_{\le}(a)\cap S_{0} \neq \varnothing$ for every $a \in S$;
		\item if $a',a'' \in V_{\le}(a)\cap S_{0}$, then $a' \,\mathcal{H}\, a''$ in $S_{0}$.
	\end{enumerate}
\end{definition}
\subsection{Rings and ideals}
A ring is a set $R$ equipped with two operations  addition $+$ and  multiplication $\cdot$ such that $(R,+)$ is an abelian group and $(R,\cdot)$ is a semigroup and multiplication distributes over addition. 
A subset $I$ of $R$ is called a \emph{left (right) ideal} of $R$ if $I$ is a subgroup of $(R,+)$ and $RI \subseteq I (IR \subseteq I)$ and a \emph{two-sided ideal} if it is both a left
ideal and a right ideal of $R$ and an ideal (left, right, or two-sided) $I$ of a ring $R$ is said to be finitely generated if there exists a finite subset $X=\{x_1,\dots,x_n\}$ of $I$ such that $I$ is generated by $X$. When $I$ is generated by a single element then it is called a principal ideal.
\begin{definition}
	A ring $R$ is called Noetherian if every ideal of $R$ is finitely generated.
\end{definition}
\begin{definition}(cf.\cite{musili})
	A ring $R$ is called a principal ideal ring (PIR)\index{Principal ideal ring} if every ideal of $R$ is principal.
\end{definition}
\subsection{Category theory}	
	A \textit{category} $ \mathcal C $ consisting of a class of  objects  written as $\nu {\mathcal C}$  and collection of morphisms $ f \in \mathcal C(A,B)$ from each object $ A = dom\; f $  to each object $ B = cod\; f $. For each pair $ (f,g) $ of morphisms with $ dom\; g=cod\; f $, a morphism $fg : \dom f \, \rightarrow\,cod\; g$ is the composition  and for each object $ A $ there exist a  unique morphism   $  1_A \in \mathcal C( A,A ) $ is  called the identity morphism on $A$. Further the composition satisfies $ f(gh) = (fg)h $ whenever defined and $  1_A f = f = f 1_B  $ for all $f \in  \mathcal C ( A,B) $.\\
The maps between categories are called functors. For categories $\mathcal{C}$ and $ \mathcal{D}  $
	a \emph{contravariant functor} $F : \mathcal{C} \to \mathcal{D}$ assigns to each object 
$A \in \nu\mathcal{C}$ an object $F(A) \in \nu\mathcal{D}$ and to every morphism 
$f : A \to B$ in $\mathcal{C}$ a morphism
$F(f) : F(B) \to F(A)$
in $\mathcal{D}$.
These assignments satisfy the following conditions:
\begin{itemize}
	\item For every object $A$ of $\mathcal{C}$,
	$	F(1_A) = 1_{F(A)}.$
	
	\item For any pair of composable morphisms $f : A \to B$ and $g : B \to C$ in $\mathcal{C}$,
	\[
	F(fg) = F(g) F(f) .
	\]
\end{itemize}
\begin{definition}
	A morphism \(f : A \to B\) is called a \emph{monomorphism}\index{Monomorphism} if, for any
	morphisms \(h',k' : D \to A\), the condition \(h'f = k'f\) implies \(h' = k'\).
	Thus monomorphisms are precisely the right-cancellable morphisms.
\end{definition}
A preorder $\mathcal {P}$ is a category such that for any $p, p' \in \nu \mathcal {P}$, the
hom-set $\mathcal {P}(p,p')$ contains at most one morphism. In this case, there is a
quasi-order relation $\preceq$ on $\nu \mathcal {P}$ defined by
\[
p \preceq p' \;\Longleftrightarrow\; \mathcal {P}(p,p') \neq \varnothing.
\]
The preorder $P$ is said to be \emph{strict} if $\preceq$ is a partial order (see cf.~\cite{kss}).
\begin{definition}(cf.\cite{kss})\label{suboject}
	Let $\mathcal{C}$ be a category and $\mathcal{P}$ a subcategory of $\mathcal{C}$.
	The pair $(\mathcal{C}, \mathcal{P})$ is called a category with subobjects\index{Subobject category}
	if the following conditions hold:
	\begin{itemize}
		\item $\mathcal{P}$ is a strict preorder with $\nu\mathcal{C} = \nu\mathcal{P}$;
		\item every morphism in $\mathcal{P}$ is a monomorphism;
		\item if $f, g \in \mathcal{P}$ and $f =  hg $ for some $h \in \mathcal{C}$,
		then $h \in \mathcal{P}$.
	\end{itemize}
	For objects $C, D \in \nu\mathcal{C}$, we denote the unique morphism in
	$\mathcal{P}$ from $C \to D$ by $j_{(C, D)}$, called the inclusion.
	In this case $C$ is referred to as a subobject of $D$.
	
\end{definition}
%\begin{definition}(cf.\cite{kss})	Let $\mathcal{C}$ be a category with subobjects.	A canonical factorization of a morphism $f$ in $\mathcal{C}$	is a factorization of the form	\[	f = qj ,	\]	where $q$ is an epimorphism and $j$ is an inclusion.\end{definition}	

\section{Ordered semigroup of ideals of a principal ideal ring.}\label{section}

In this section, we establish that the set of all ideals of a commutative principal ideal ring carries a natural structure of a regular ordered semigroup, where the semigroup operation is ideal multiplication and the partial order is determined by the divisibility of generators. This order is compatible with the multiplication and every element of the semigroup is ordered regular.

Let $R$ be a commutative principal ideal ring with unity and $\mathcal{I}(R)$ the set of all ideals of $R$. 
Then for every $A \in \mathcal{I}(R)$, there exists $a \in R$ such that
\[
A = \langle a \rangle = \{ra \mid r \in R\}.
\]
For $A, B \in \mathcal{I}(R)$, the product of ideals $A$ and $B$ is defined by
\[
AB = \left\{ \sum_{i=1}^{n} a_i b_i \;\middle|\; a_i \in A,\ b_i \in B,\ n \in \mathbb{N} \right\},
\]
clearly, equipped with this operation $\mathcal{I}(R)$ is a semigroup.
For, consider principal ideals $A = \langle a \rangle$ and $B = \langle b \rangle$, then
\[
\langle a \rangle \langle b \rangle 
= \left\{ \sum_{i=1}^{n} r_i a\, s_i b \;\middle|\; r_i, s_i \in R,\ n \in \mathbb{N} \right\}.
\]
When $R$ is commutative the product of principal ideals simplifies to
\[
AB = \langle a \rangle \cdot \langle b \rangle = \langle ab \rangle.
\]
Indeed, for any $r \in R$, we have $r(ab) = (ra)b \in \langle a\rangle\langle b\rangle$, which shows that
$\langle ab\rangle \subseteq \langle a\rangle\langle b\rangle$.
On the other hand, each element
$ x \in \langle a\rangle\langle b\rangle $
can be written as a finite sum $\sum_{i=1}^n (r_i a)(s_i b)$ with $r_i,s_i \in R$,
 hence $x \in \langle ab\rangle$, and,
$\langle a\rangle\langle b\rangle = \langle ab\rangle.$

For $A  = \langle a \rangle, B = \langle b \rangle \in \mathcal{I}(R)$, define the order relation $\preccurlyeq$ as follows: 
\[
A \preccurlyeq B \iff a \mid b \;\text{in}\; R,
\]
then $\preccurlyeq$ is a partial order on $\mathcal{I}(R)$, and is compatible with multiplication. For,  $A = \langle a \rangle$, since $a \mid a, \quad a = a \cdot 1$ and so $A \preccurlyeq A$.
Suppose $A \preccurlyeq B$ and $B \preccurlyeq A$, then $a \mid b$ and $b \mid a$, ie., there exist
$x,y \in R$ such that $b=ax$ and $a=by$. It follows that
$\langle a\rangle \subseteq \langle b\rangle$ and
$\langle b\rangle \subseteq \langle a\rangle$ and hence
$\langle a\rangle=\langle b\rangle$, and so $A =B$.
When $A \preccurlyeq B$ and $B \preccurlyeq C$, where $C = \langle c \rangle$, then $a \mid b$ and $b \mid c$, so there exist $x, y \in R$ such that
$b = ax $ and $ c = by,$  hence $a \mid c$ and therefore  $A \preccurlyeq C$, ie., the relation 
$\preccurlyeq$ is a partial order on $\mathcal{I}(R)$.
Also since $R$ is commutative, it easily follows that $\preccurlyeq$ is compatible with the semigroup multiplication on $\mathcal{I}(R)$.

Finally, we establish ordered regularity. Let $A \in \mathcal{I}(R)$. Then there exists $a \in R$ such that $A = \langle a \rangle$. 
For any $x \in R$, we have $a \mid axa$. Let $X = \langle x \rangle$. Then $
\langle axa \rangle = \langle a \rangle \langle x \rangle \langle a \rangle = AXA.$
Thus $A \preccurlyeq AXA$.
Therefore, each element of $\mathcal{I}(R)$ is ordered regular. Hence the triple $(\mathcal{I}(R), \cdot, \preccurlyeq)$ forms a regular ordered semigroup.
The preceding results are collected in the following theorem.

\begin{theorem}
	Let $R$ be a commutative principal ideal ring with unity and $\mathcal{I}(R)$ denote the set of all ideals of $R$. 
	Define a binary operation on $\mathcal{I}(R)$ by ideal multiplication and define a relation $\preccurlyeq$ on $\mathcal{I}(R)$ by for  $A  = \langle a \rangle, B = \langle b \rangle \in \mathcal{I}(R)$
	\[
	A \preccurlyeq B \iff a \mid b \; \text{in}\; R,
	\]
	Then $\big(\mathcal{I}(R), \cdot, \preccurlyeq\big)$ is a regular ordered semigroup.
\end{theorem}
\begin{remark}
	Suppose that $A, B \in \mathcal{I}(R)$ admit multiple generators, say 
	$A = \langle a \rangle = \langle a' \rangle$ and 
	$B = \langle b \rangle = \langle b' \rangle$. 
	If $a \mid b$ then $a' \mid b'$.  Suppose that $a \mid b$, then $b = ra$ for some $r \in R$, since $b' \in \langle b \rangle$, there exists $x \in R$ such that $b' = xb$. 
	Also, $a \in \langle a' \rangle$, implies there exists $y \in R$ such that $a = ya'$ and hence,
	\[
	b' = xb = x(ra) = xr(ya') = (xry)a',
	\]
	which shows that $a' \mid b'$.
\end{remark}
\begin{remark} When the ring $R$ is not commutative 
$\langle a\rangle \langle b\rangle = \langle ab\rangle$
does not  hold. 
\end{remark}
\begin{example}
	Consider the ring 
	\[
	R=\left\{ 
	\begin{pmatrix}
		a & b\\
		0 & c
	\end{pmatrix}
	\;\middle|\; a,b,c \in F
	\right\}
	\]
	of all upper triangular $2\times 2$ matrices over a field $F$ and the two sided ideals
	\[
	I_{1}
	= \left\langle 
	\begin{pmatrix}
		1 & 0\\
		0 & 0
	\end{pmatrix}
	\right\rangle
	= \left\{
	\begin{pmatrix}
		a & b\\
		0 & 0
	\end{pmatrix}
	\;\middle|\; a,b \in F
	\right\}
	\]
	and
	\[
	I_{2}
	= \left\langle 
	\begin{pmatrix}
		0 & 0\\
		0 & 1
	\end{pmatrix}
	\right\rangle
	= \left\{
	\begin{pmatrix}
		0 & c\\
		0 & d
	\end{pmatrix}
	\;\middle|\; c,d \in F
	\right\}, 
	\]
	
	then
	\[
	I_{1} I_{2}
	= \left\langle 
	\begin{pmatrix}
		1 & 0\\
		0 & 0
	\end{pmatrix}
	\right\rangle
	\left\langle 
	\begin{pmatrix}
		0 & 0\\
		0 & 1
	\end{pmatrix}
	\right\rangle
	= \left\{
	\begin{pmatrix}
		0 & a\\
		0 & 0
	\end{pmatrix}
	\;\middle|\; a \in F
	\right\}
	= \left\langle
	\begin{pmatrix}
		0 & 1\\
		0 & 0
	\end{pmatrix}
	\right\rangle.
	\]
	
	On the other hand,
	\[
	\left\langle
	\begin{pmatrix}
		1 & 0\\
		0 & 0
	\end{pmatrix}
	\begin{pmatrix}
		0 & 0\\
		0 & 1
	\end{pmatrix}
	\right\rangle
	=
	\left\langle
	\begin{pmatrix}
		0 & 0\\
		0 & 0
	\end{pmatrix}
	\right\rangle,
	\]
	
	hence,
	\[
	\left\langle
	\begin{pmatrix}
		1 & 0\\
		0 & 0
	\end{pmatrix}
	\right\rangle
	\left\langle
	\begin{pmatrix}
		0 & 0\\
		0 & 1
	\end{pmatrix}
	\right\rangle
	\neq
	\left\langle
	\begin{pmatrix}
		1 & 0\\
		0 & 0
	\end{pmatrix}
	\begin{pmatrix}
		0 & 0\\
		0 & 1
	\end{pmatrix}
	\right\rangle.
	\]
\end{example}

Next we investigate some properties of regular ordered semigroups of ideals of commutative principal ideal rings.

\begin{theorem}
	Let $R$ and $S$ be commutative principal ideal rings with unity.
	If $R$ and $S$ are isomorphic then the semigroups
	$(\mathcal{I}(R), \cdot , \preccurlyeq_{R})$ and
	$(\mathcal{I}(S), \cdot , \preccurlyeq_{S})$ are isomorphic.
\end{theorem}

\begin{proof}
	Let $\phi \colon R \to S$ be a ring isomorphism. Define a map 
	$\psi \colon \mathcal{I}(R) \to \mathcal{I}(S)$ by
	\[
	\psi(A) = \langle \phi(a) \rangle , \quad \text{for } A = \langle a \rangle \in \mathcal{I}(R).
	\]
	To show that $\langle \phi(a) \rangle = \phi(\langle a \rangle)$, consider 
	$y \in \langle \phi(a) \rangle$  then $y = s\,\phi(a)$ for some $s \in S$. 
	Since $\phi$ is surjective, there exists $r \in R$ such that $s = \phi(r)$ and so 
	$y = \phi(r)\phi(a) = \phi(ra).$ Also since $ra \in \langle a \rangle$, we have $y \in \phi(\langle a \rangle)$  thus $\langle \phi(a) \rangle \subseteq \phi(\langle a \rangle)$.
	
	Conversely, for $x \in \phi(\langle a \rangle),\quad x = \phi(ra)$ for some $r \in R$, and so $	x = \phi(r)\phi(a),$
	which implies $x \in \langle \phi(a) \rangle$. Thus 
	$\phi(\langle a \rangle) \subseteq \langle \phi(a) \rangle$, and therefore
$	\langle \phi(a) \rangle = \phi(\langle a \rangle).$
	
	Next, suppose $\psi(A) = \psi(B)$, then
$	\langle \phi(a) \rangle = \langle \phi(b) \rangle$
	implies $\phi(\langle a \rangle) = \phi(\langle b \rangle)$. 
	Since $\phi$ is injective it follows that $\langle a \rangle = \langle b \rangle$, 
	so $A = B$, ie., $\psi$ is injective. For $T = \langle t \rangle \in \mathcal{I}(S)$, 
	since $\phi$ is surjective, there exists $d \in R$ such that $t = \phi(d)$
	showing that $\psi$ is surjective. and hence $\psi$ is bijective.
	
	To verify that $\psi$ is a semigroup homomorphism, consider $A = \langle a \rangle,\, B = \langle b \rangle\, \in \mathcal{I}(R)$; 
	\[
	\psi(AB)
	= \psi(\langle a \rangle \langle b \rangle)
	= \psi(\langle ab \rangle)
	= \langle \phi(ab) \rangle
	= \langle \phi(a)\phi(b) \rangle
	= \psi(A)\psi(B),
	\]
	and so $\phi(ab) = \phi(a)\phi(b)$.
	Finally, $\psi$ preserves the order, for $A = \langle a \rangle\, B = \langle b \rangle,\, \in \mathcal{I}(R)$ such that 
	$A \preccurlyeq_R B$. Then $a \mid b$, so $b = ax$ for some $x \in R$, hence $\phi(b) = \phi(ax) = \phi(a)\phi(x),$
	which shows that $\phi(a) \mid \phi(b)$, ie., $	\langle \phi(a) \rangle \preccurlyeq_S \langle \phi(b) \rangle,$ implies, $\psi(A) \preccurlyeq_S \psi(B).$
	
	Therefore, $\psi$ is a bijective order-preserving semigroup homomorphism, consequently,
	\[
	(\mathcal{I}(R), \cdot, \preccurlyeq_R)
	\cong
	(\mathcal{I}(S), \cdot, \preccurlyeq_S).
	\]
\end{proof}

\begin{remark}
The reverse implication of the preceding theorem need not hold in general.
To illustrate, consider the fields $\mathbb{R}$ and $\mathbb{Q}$, each one is a commutative principal ideal ring with identity whose only ideals of are
$\langle 0 \rangle $ and $\langle 1 \rangle$.
As a consequence the associated ordered semigroups of ideals are isomorphic
\[
(\mathcal{I}(\mathbb{R}), \cdot, \preccurlyeq_{\mathbb{R}})
\cong
(\mathcal{I}(\mathbb{Q}), \cdot, \preccurlyeq_{\mathbb{Q}})
\]
nevertheless the rings $\mathbb{R}$ and $\mathbb{Q}$ themselves are not isomorphic.
\end{remark}
An inverse transversal of the ordered semigroup of ideals of $R$ is identified below.

\begin{proposition}
	Let $R$ be a commutative principal ideal ring with unity and 
	$\big(\mathcal{I}(R), \cdot, \preccurlyeq\big)$ be the ordered semigroup of
	ideals. Then $\{ \langle 1\rangle \}$ is an inverse transversal of
	$\big(\mathcal{I}(R), \cdot, \preccurlyeq\big)$.
\end{proposition}

\begin{proof}
	Consider $ A = \langle a\rangle \in \mathcal{I}(R)$. Since $R = \langle 1 \rangle$ and 
	$a = a \cdot 1 \cdot a$, it follows that $a \mid a1a$, hence $A \preccurlyeq ARA.$
	Similarly, $1 \mid 1a1$, which implies $R \preccurlyeq RAR,$ thus $R$ is as an ordered inverse for every element of $\mathcal{I}(R)$. 	
	Moreover, $	RR = R \quad \text{and} \quad R \preccurlyeq R,$
	so $R$ forms an inverse ordered subsemigroup of $\big(\mathcal{I}(R), \cdot, \preccurlyeq\big)$. Clearly $(R] = \{R\},$ and all the conditions for an inverse transversal (see Definition~\ref{definvtrsl}) are satisfied. Therefore,
$	\mathcal{I}(R)_0 = \{R\}.$
\end{proof}

We next examine the connection between regularity of commutative
principal ideal ring and regularity of its semigroup of ideals.

\begin{proposition}
	Let $R$ be a commutative principal ideal ring with unity and $a \in R$.
	Then $a$ is von Neumann regular in $R$ if and only if
	$\langle a \rangle$ is a regular\index{Regular} element of the semigroup
	$\mathcal{I}(R)$.
\end{proposition}
\begin{proof}
	Let $a$ be a von Neumann regular element of $R$. Then there exists $x \in R$ such that
$	a = axa.$
	Consequently $	\langle a\rangle = \langle axa\rangle
	= \langle a\rangle \langle x\rangle \langle a\rangle$
	which shows $\langle a\rangle$ is a regular element of the semigroup
	$\mathcal{I}(R)$.\\
		
	\noindent
	Conversely suppose that $\langle a\rangle$ is regular in $\mathcal{I}(R)$.
	Then there exists $x \in R$ satisfying
$	\langle a\rangle
	= \langle a\rangle \langle x\rangle \langle a\rangle
	= \langle axa\rangle.$
	Since $R$ is a commutative principal ideal ring, equality of principal ideals yields
$	a = r(axa) = a(rx)a
	\quad \text{for some } r \in R.$
Hence $a$ is von Neumann regular in $R$.
\end{proof}
In the following theorem we investigate Green’s relations\index{Green’s relations}
on the ordered semigroup $\big(\mathcal{I}(R), \cdot , \preccurlyeq\big)$.
\begin{theorem}
	Let $R$ be a commutative principal ideal ring with unity and 
	$\big(\mathcal{I}(R), \cdot , \preccurlyeq\big)$  the ordered semigroup of ideals of $R$.
	Then Green’s relations $\mathcal{L}, \mathcal{R}, \mathcal{H}, \mathcal{D}$, and $\mathcal{J}$ 	on $\mathcal{I}(R)$ are all universal relations.
\end{theorem}

\begin{proof}
	The existence of the zero ideal together with the
	commutativity of $R$ forces all Green’s relations on
	$\mathcal{I}(R)$ to be universal. For, consider the ordered semigroup
	$\big(\mathcal{I}(R), \cdot , \preccurlyeq\big)$, since $a \mid 0$ for every $a \in R$, it follows that
$\langle a \rangle \preccurlyeq \langle 0 \rangle \quad \text{for all } a \in R.$
Thus for every $A \in \mathcal{I}(R)$, we have $A \preccurlyeq 0$, where $0 = \langle 0 \rangle$.
	
Let $(\mathcal{I}(R)A]$ denote the ideal of the ordered semigroup $\mathcal{I}(R)$ generated by $A$. Since $0 \in \mathcal{I}(R)$ and
$0A = 0,$ 	it follows that $0 \in (\mathcal{I}(R)A]$. 
	As ideals in an ordered semigroup are downward closed with respect to $\preccurlyeq$ and since $A \preccurlyeq 0$ for all $A \in \mathcal{I}(R)$, we obtain
$(\mathcal{I}(R)0] = \mathcal{I}(R).$
	Hence, the ordered semigroup $\big(\mathcal{I}(R), \cdot , \preccurlyeq\big)$ admits no proper ideals.
	Since $R$ is commutative the semigroup $\big(\mathcal{I}(R), \cdot\big)$ is also commutative, therefore, all Green’s relations coincide; that is,
	\[
	\mathcal{L} = \mathcal{R} = \mathcal{J} = \mathcal{H} = \mathcal{D}
	\quad \text{on } \mathcal{I}(R).
	\]
	Moreover, for any $A \in \mathcal{I}(R)$, the corresponding Green’s class is the entire semigroup:
	\[
	\mathcal{L}_{A}
	= \mathcal{R}_{A}
	= \mathcal{J}_{A}
	= \mathcal{H}_{A}
	= \mathcal{D}_{A}
	= \mathcal{I}(R).
	\]
Therefore, all Green’s relations on $\big(\mathcal{I}(R), \cdot , \preccurlyeq\big)$ are universal.
\end{proof}
Next we provide certain regularity conditions on $\big(\mathcal{I}(R), \cdot , \preccurlyeq\big)$.
\begin{theorem}
	Let $R$ be a commutative principal ideal ring with unity.  
	Then the ordered semigroup
	$\big(\mathcal{I}(R), \cdot , \preccurlyeq \big)$ satisfies the following properties:
	\begin{enumerate}
		\item it is intra-regular
		\item it is an inverse ordered semigroup
		\item every element is an ordered idempotent
		\item it is completely regular
		\item it is group-like
		\item it is a Clifford ordered semigroup.
	\end{enumerate}
\end{theorem}
 
\begin{proof}
	\begin{enumerate}
		\item Consider $a, x, y \in R$, we have $xa^{2}y = a(xay)$, hence $a \mid xa^{2}y$. Therefore, $A \preccurlyeq XA^{2}Y,$
	where $A = \langle a \rangle$, $X = \langle x \rangle$, and $Y = \langle y \rangle$.
	Thus $\big(\mathcal{I}(R), \cdot , \preccurlyeq\big)$ is intra-regular.

	\item For any $a,a' \in R$, we have $a \mid a a' a$ and $a' \mid a' a a'$.
	Consequently, $A \preccurlyeq A A'A
	\quad \text{and} \quad
	A'\preccurlyeq
	A'AA'$, where $A = \langle a \rangle, A' = \langle a' \rangle$. Thus every element is an ordered inverse of every other element and
	$\mathcal{I}(R)$ consists of a single $\mathcal{H}$-class.
	Hence it is an inverse ordered semigroup.
	
	\item Since $a \mid a^{2}$ for all $a \in R$, we obtain $A \preccurlyeq  A^{2}$, where $A = \langle a \rangle$, therefore every element of $\mathcal{I}(R)$ is an ordered idempotent.
	
	\item For all $x \in R$, we have $a \mid a^{2} x a^{2}$ which yields $A \preccurlyeq A^{2}X A^{2}$, where $A = \langle a \rangle, X = \langle x \rangle $.
	Hence $\big(\mathcal{I}(R), \cdot , \preccurlyeq\big)$ is completely regular.
	
	\item For $a,b \in R$, since $a \mid ab$ and $a \mid ba$, we have $A \preccurlyeq AB$ and $A \preccurlyeq BA$, where $A = \langle a \rangle $ and $B = \langle b \rangle$,
	showing that the ordered semigroup is group-like.
	
	\item Since every element of $\mathcal{I}(R)$ is an ordered idempotent
	$ E = \langle e \rangle \in \mathcal{I}(R)$.
	For any $a \in R$, if $ae \mid eua$ and $ea \mid ave$ for all $u,v \in R$, then $AE \preccurlyeq EUA $ and $EA \preccurlyeq AVE $, where $A = \langle a \rangle$, $ U = \langle u \rangle$ and $V = \langle v \rangle$.
\end{enumerate}
	Therefore
	$\big(\mathcal{I}(R), \cdot , \preccurlyeq\big)$ is a Clifford ordered semigroup.
\end{proof}

\begin{example} Ordered semigroup of ideals  $(\mathcal{I}(\mathbb{Z}), \cdot , \preccurlyeq)$; of $\mathbb{Z}$.\\
	Clearly $n$ and $-n$ generate the same principal ideal in $\mathbb{Z}$, thus 
	\[
	\mathcal{I}(\mathbb{Z}) = \{\, \langle n \rangle \mid n \in \mathbb{N} \cup \{0\} \,\}
	\]
	is a semigroup under ideal multiplication.
	In particular for any $\langle m \rangle, \langle n \rangle \in \mathcal{I}(\mathbb{Z})$,
	their product is given by
$	\langle m \rangle \cdot \langle n \rangle = \langle mn \rangle$
	and the order relation $\preccurlyeq$ on $\mathcal{I}(\mathbb{Z})$ by
	\[
	\langle n \rangle \preccurlyeq \langle m \rangle
	\quad \text{if and only if} \quad n \mid m \; \text{in}\, \mathbb{Z}.
	\]
	Let $\mathbb{N}^{0} = \mathbb{N} \cup \{0\}$ and define a mapping
	\[
	\varphi : \mathcal{I}(\mathbb{Z}) \longrightarrow \mathbb{N}^{0},
	\quad 	\varphi(\langle n \rangle) = n.
	\]
	It is straightforward to verify that $\varphi$ is a bijection preserving
	both the semigroup operation and the order relation.
	Hence
	\[
	(\mathcal{I}(\mathbb{Z}), \cdot , \preccurlyeq)
	\cong
	(\mathbb{N}^{0}, \cdot , \mid).
	\]
\end{example}
\begin{example} Ordered semigroup of ideals $(\mathcal{I}(\mathbb{Z}_{n}), \cdot , \preccurlyeq)$ of $\mathbb{Z}_{n}$.\\
	Each  ideal of the ring $\mathbb{Z}_n$ of integers modulo $n$  is generated by a divisor of $n$ (see\cite{punin}). Consequently
	\[
	\mathcal{I}(\mathbb{Z}_n)
	=
	\{\, \langle d \rangle \mid d \text{ is a  divisor of } n \,\}.
	\]
	For any $a \in \mathbb{Z}_n$, $	\langle a \rangle = \langle \gcd(a,n) \rangle $ and for 
	$\langle d_1 \rangle, \langle d_2 \rangle \in \mathcal{I}(\mathbb{Z}_n)$, their product under ideal multiplication is given by
	\[
	\langle d_1 \rangle \cdot \langle d_2 \rangle
	= \langle d_1 d_2 \rangle
	= \langle \gcd(d_1 d_2, n) \rangle.
	\]
	Define a partial order $\preccurlyeq$ on $\mathcal{I}(\mathbb{Z}_n)$ by
	\[
	\langle d_1 \rangle \preccurlyeq \langle d_2 \rangle
	\quad \text{if and only if} \quad
	d_1 \mid d_2 \; \text{ in}\; \mathbb{Z}_n,
	\]
	equivalently $\langle d_1 \rangle \preccurlyeq \langle d_2 \rangle$ if and only if
	$d_2 = m d_1$ for some integer $m$. Since both $d_1$ and $d_2$ divide $n$, the
	integer $m$ is also a divisor of $n$.
	Let $D(n)$ denote the set of all positive divisors of $n$.
	Define a binary operation $*$ on $D(n)$ by
	\[
	d_1 * d_2 = \gcd(d_1 d_2, n),
	\quad \text{for all } d_1, d_2 \in D(n).
	\]
	The divisibility relation defines a partial order on $D(n)$ and this order is compatible with the operation $*$. Thus $(D(n), *, \mid)$ is an ordered semigroup.
	
	The a map
	\[
	\varphi : \mathcal{I}(\mathbb{Z}_n) \longrightarrow D(n),
	\qquad
	\varphi(\langle d \rangle) = d.
	\]
	is well defined, bijective and preserves both the semigroup operation
	and the order relation.
	Hence,
	\[
	(\mathcal{I}(\mathbb{Z}_n), \cdot , \preccurlyeq)
	\cong
	(D(n), *, \mid).
	\]

\end{example}

\begin{example}  Ordered semigroup $(\mathcal{I}(F[x]), \cdot , \preccurlyeq)$ of ideals of $F[x]$.\\

Consider the polynomial ring $F[x]$, over a field $F$. Then
\[
\mathcal{I}(F[x])
=
\{\, \langle f(x) \rangle \mid f(x) \in F[x] \,\}.
\]
is an ideal where the 
ideal multiplication is given by
$\langle f(x) \rangle \cdot \langle g(x) \rangle
= \langle f(x) g(x) \rangle $, for any $\langle f(x) \rangle, \langle g(x) \rangle \in \mathcal{I}(F[x])$
and the order relation $\preccurlyeq$ on $\mathcal{I}(F[x])$ by
\[
\langle f(x) \rangle \preccurlyeq \langle g(x) \rangle
\quad \text{if and only if} \quad
f(x) \mid g(x)\; \text{in}\; F[x]. 
\]
Observe that for any  $ a \in F^{*}  =F \setminus\{0\} $, then 
$\langle a \rangle = F[x]$, while $\langle 0 \rangle = \{0\}$ is the zero
ideal. Moreover two principal ideals $\langle f(x) \rangle$ and
$\langle g(x) \rangle$ coincide if and only if $f(x)$ and $g(x)$ differ by a
nonzero scalar multiple. 
The relation $\sim$ on  $F[x]$, defined by
\[
p(x) \sim q(x)
\iff
p(x) = a q(x) \quad \text{for some } a \in F^{*}.
\]
is an equivalence relation on $F[x]$. and $F[x]/\sim
= \{\, [f(x)] \mid f(x) \}$
denote the set of equivalence classes of $\sim$.
For $[f(x)], [g(x)] \in F[x]/\sim$, define multiplication by
$[f(x)] [g(x)] = [f(x) g(x)].$
Then $F[x]/\sim$ is a semigroup and we introduce a partial order on $F[x]/\sim$ by setting
$[f(x)] \mid [g(x)]
\iff
f(x) \mid g(x) \text{ in }F[x].$
Now the map
\[
\psi:\mathcal{I}(F[x])\longrightarrow F[x]/\sim
\quad \text{by} \quad
\psi(\langle f(x)\rangle)=[f(x)].
\]
is well defined, bijective and preserves both multiplication and
order. Consequently
\[
(\mathcal{I}(F[x]),\cdot,\preccurlyeq)
\cong
(F[x]/\sim,\cdot,\mid).
\]
\end{example}
\section{Category of ideals of a principal ideal ring}\label{pir}

In \cite{calcutta}, we introduced the categories $\mathbb{L}_R$ and $\mathbb{R}_R$, 
whose objects are the left and right ideals respectively of a Noetherian ring $R$ with unity and morphisms left and right $R$-linear maps and, it is  shown that these are 
preadditive categories with a zero object which can be viewed as full subcategories of the category of $R$-modules. Further, these are caategories with subobjects and their morphisms admit factorization. 

In the following, the ring $R$  we discuss is a commutative principal ideal ring with unity, 
so every ideal of $R$ is generated by a single element, hence $R$ is Noetherian. Since $R$ is commutative the notions of left and right ideals coincide, therefore, the categories $\mathbb{L}_R$ and $\mathbb{R}_R$ are identical and we denote this common category by
$\mathbb{I}_R$.  
Each objects in $\mathbb{I}_R$ has the form $A = \langle a \rangle$ for some $a \in R$ and any element in $A$ is written as $ra$ for some $r \in R$. For ideals $A = \langle a \rangle$ and $B = \langle b \rangle$, the set of morphisms $\mathbb{I}_R(A,B)$ consist of all $R$-linear maps from $A$ to $B$, thus, a function $f : A \to B$ is a morphism if and only if
\[
f(r_1 a + r_2 a) = (r_1 + r_2) f(a)
\quad \text{for all } r_1, r_2 \in R,
\]
since each ideal is generated by a single element any such morphism is completely determined by the image of the generator.  
Also $f : A \to B$ is $R$-linear, and $f(a) \in B = \langle b \rangle$ implies there exists $c \in R$ such that $f(a) = cb$, hence, for every $ra \in A,\quad f(ra) = rcb.$
Thus, every morphism from $A$ to $B$ is uniquely determined by a choice of $c \in R$, 
further, being a Noetherian ring, the category $\mathbb{L}_R$ admits subobjects and the subobject relation in $\mathbb{L}_R$ given by the inclusion of left ideals. 

Next we describe the category $\mathbb{I}_R$,  when $R$ is a commutative principal ideal ring and proceed to show that it is a category with subobjects. For $A = \langle a \rangle,\, B = \langle b \rangle \in \mathbb{I}_R$,  define
\[
A \subseteq B \iff a = rb \text{ for some } r \in R
\]
which is clearly a partial order on objects of $\mathbb{I}_R$.
For $A \subseteq B$, the map $j_{(A,B)} : A \to B$ given by  
\[
j_{(A,B)}(ra) = ra, 
\] 
is $R$-linear and hence a morphism in $\mathbb{I}_R$. Moreover, it is a monomorphism and termed it as the inclusion of $A$ into $B$.
Consider the subcategory $\mathcal{P}$ of $\mathbb{I}_R$ having the same objects as $\mathbb{I}_R$ and whose morphisms are precisely these inclusion maps $j_{(A,B)}$, since there is at most one such morphism between any two objects, $\mathcal{P}$ is a strict preorder.
Finally, suppose that $j_{(A,D)}$ and $j_{(B,D)}$ are morphisms in $\mathcal{P}$ and that
\[
j_{(A,D)} = h \, j_{(B,D)}
\]
for some morphism $h : A \to B$ in $\mathbb{I}_R$, as both $j_{(A,D)}$ and $j_{(B,D)}$ are inclusion maps into $D = \langle d \rangle$, it follows that $h$ must coincide with the inclusion map from $A$ into $B$. 
Therefore, the pair $(\mathbb{I}_R, \mathcal{P})$ satisfies the axioms of a category with subobjects.

\section{Correspondence of ideal categories and ordered semigroup structures of  PIR}

In the following we explore the connection between categorical  and ordered semigroup structures arising from the ideals of a commutative principal ideal ring.

Consider the commutative principal ideal ring $R$ with unity. It is already seen that $\mathbb{I}_R$, the category of ideals of $R$ is a category with subobjects and  $\mathcal{I}(R)$ equipped with ideal multiplication and order relation $\preccurlyeq$ defined by divisibility of generators is an ordered semigroup.

To formalize the relation between these structures,  we construct a category $\mathcal{C}$ 
whose objects are those of $\mathcal{I}(R)$, and for $A = \langle a \rangle,\, B = \langle b \rangle \in \mathcal{I}(R) $, define a morphism
\[
A \to B
\quad \text{if and only if} \quad
A \preccurlyeq B,
\]
i.e., whenever $a \mid b$ denoted by $f_{(A,B)}$. Since the divisibility relation determines the existence of a morphism there can be at most one morphism between any two objects.
The composition in $\mathcal{C}$ is induced by the transitivity of divisibility, that is when 
$f_{(A,B)} : A \to B
\quad \text{and} \quad
f_{(B,C)} : B \to C,$
then $a \mid b$ and $b \mid c$, which implies $a \mid c$, accordingly, 
$
f_{(A,B)}\, f_{(B,C)} = f_{(A,C)}.$
 For each object $A = \langle a \rangle$, the relation $a \mid a$ yields the identity morphism $f_{(A,A)}$ . The identity laws follow immediately from the definition of composition and associativity of composition is a direct consequence of the transitivity of divisibility. Thus, $\mathcal{C}$ is a category in which there is at most one morphism between any two objects; ie., it is a strict preorder category.

Now, consider the map
$F : \mathcal{C} \to \mathbb{I}_R$, whose object mad is identity and for a morphism 
$f_{(A,B)} : A \to B$ in $\mathcal{C}$, define
$F(f_{(A,B)}) = j_{(B,A)},$
where $j_{(B,A)} : B \to A$ is the inclusion map in $\mathbb{I}_R$. It can be seen that $F$ is a contravariant functor. For, consider any object $A$, we have
$F(f_{(A,A)}) = j_{(A,A)},$
which is the identity morphism on $A$ in $\mathbb{I}_R$.
Next, consider morphisms $f_{(A,B)} : A \to B,\, f_{(B,C)} : B \to C$, then
$F(f_{(A,B)}) = j_{(B,A)}, \quad F(f_{(B,C)}) = j_{(C,B)}$, and their composition in $\mathbb{I}_R$ is
$j_{(C,B)} \, j_{(B,A)} = j_{(C,A)},$
which is the inclusion of $C$ into $A$. On the other hand,
$F(f_{(A,B)} \circ f_{(B,C)}) = F(f_{(A,C)}) = j_{(C,A)}$, thus,
\[
F(f_{(A,B)} \, f_{(B,C)}) = F(f_{(B,C)}) \, F(f_{(A,B)}),
\]
showing that composition is reversed.

Therefore, we conclude that there is a natural correspondence between the ordered semigroup structure $\mathcal{I}(R)$ and the categorical structure $\mathbb{I}_R$ given by the contravariant functor $F$ from $\mathcal{C}$ to $\mathbb{I}_R$.

\end{document}